\documentclass[11pt,reqno]{amsart}
\usepackage{amsmath,amssymb,bbm,subfig,url}

\usepackage{tikz}
\usetikzlibrary{calc}
\usetikzlibrary{arrows}
\usetikzlibrary{patterns}
\usetikzlibrary{through}
\usetikzlibrary{plotmarks}

\newtheorem{theorem}{Theorem}[section]
\newtheorem{lemma}[theorem]{Lemma}
\newtheorem{proposition}[theorem]{Proposition}
\newtheorem{corollary}[theorem]{Corollary}

\newtheorem*{definition}{Definition}

\numberwithin{equation}{section}

\newcommand{\D}{{\mathbb D}}

\begin{document}

\title[]{Minimizing Neumann fundamental tones of triangles: \\
an optimal Poincar\'{e} inequality}

\author[]{R. S. Laugesen and B. A. Siudeja}
\address{Department of Mathematics, University of Illinois, Urbana,
IL 61801, U.S.A.} \email{Laugesen\@@illinois.edu}
\email{Siudeja\@@illinois.edu}
\date{\today}


\begin{abstract}
The first nonzero eigenvalue of the Neumann Laplacian is shown to
be minimal for the degenerate acute isosceles triangle, among all
triangles of given diameter. Hence an optimal Poincar\'{e} inequality for triangles is derived.

The proof relies on symmetry of the Neumann fundamental mode for isosceles
triangles with aperture less than $\pi/3$. Antisymmetry is proved
for apertures greater than $\pi/3$.
\end{abstract}

\maketitle

\section{\bf Introduction}

Payne and Weinberger \cite{PW60} proved for arbitrary convex
domains that
\[
\text{$\mu_1 D^2$ is minimal for the degenerate rectangular box,}
\]
where $\mu_1$ is the first nonzero eigenvalue of the Neumann
Laplacian and $D$ is the diameter of the domain. Our main result is a stronger inequality for triangular domains in the plane:
\[
\text{$\mu_1 D^2$ is minimal for the degenerate acute isosceles
triangle.}
\]
We prove our result by first stretching to an isosceles triangle and then
bisecting and stretching repeatedly to approach the degenerate
case. Payne and Weinberger's method of thinly slicing an arbitrary
domain does not apply, since the slices would not be triangular.

A corollary is an optimal Poincar\'{e} inequality for triangles, namely that
\[
\int_T v^2 \, dA < \frac{D^2}{j_{1,1}^2} \int_T |\nabla v|^2 \, dA
\]
whenever the function $v$ has mean value zero over the triangle
$T$. Here $j_{1,1} \simeq 3.8317$ denotes the first positive root of the Bessel function $J_1$.

Our proof relies on symmetry properties of isosceles triangles.
We show by our ``Method of the Unknown Trial Function''
(Section~\ref{isec3}) that the first nonconstant Neumann mode of
an isosceles triangle is symmetric when the aperture of the
triangle is less than $\pi/3$. We similarly prove antisymmetry
when the aperture exceeds $\pi/3$. In that case the nodal curve
lies on the shortest altitude. 

Our companion paper \cite{LS09a} \emph{maximizes} $\mu_1$ among
triangles, under perimeter or area normalization, with the minimizer being
equilateral. We know of no other papers in the literature that
study sharp isoperimetric type inequalities for Neumann
eigenvalues of triangles. Note the Neumann eigenfunctions of triangles were investigated for the ``hot spots'' conjecture, by Ba\~{n}uelos and Burdzy \cite{BaBu}, and the approximate location of the nodal curve for non-isosceles triangles was studied in recent work of Atar and Burdzy \cite{AtBu}, using probabilistic methods. 

Dirichlet eigenvalues of triangles have received considerable
attention \cite{AF06,F06,fresiu,S07,S09}. Particularly interesting is the Dirichlet gap conjecture for triangles, due to Antunes and Freitas \cite{AF08}, which claims $(\lambda_2 - \lambda_1)D^2$ is minimal for the equilateral triangle; some progress has been made recently by Lu and Rowlett \cite{LR08}. Dirichlet eigenvalues of degenerate domains have also
been investigated lately \cite{BF09,F07}. Some of these Dirichlet triangle results are discussed in our companion paper \cite[Section~10]{LS09a}.

For broad surveys of isoperimetric eigenvalue
inequalities, see the paper by Ashbaugh \cite{A99}, and the
monographs of Bandle \cite{B79}, Henrot \cite{He06}, Kesavan
\cite{K06} and P\'{o}lya--Szeg\H{o} \cite{PS51}.

\section{\bf Notation}
\label{notation}

The Neumann eigenfunctions of the Laplacian on a bounded plane
domain $\Omega$ with Lipschitz boundary satisfy $-\Delta u = \mu u$ with
natural boundary condition $\partial u / \partial n = 0$. The
eigenvalues $\mu_j$ are nonnegative, with
\[
0 = \mu_0 < \mu_1 \leq \mu_2 \leq \dots \to \infty .
\]
Call $\mu_1$ the \textbf{fundamental tone}, since
$\sqrt{\mu_1}$ is proportional to the lowest frequency of vibration
of a free membrane over the domain. Call the eigenfunction $u_1$ a
\textbf{fundamental mode}.

The Rayleigh Principle says
\[
\mu_1 = \min_{\int_\Omega v \, dA = 0} R[v]
\]
where 
\[
R[v] = \frac{\int_\Omega |\nabla v|^2 \, dA}{\int_\Omega v^2 \, dA}
\]
is the Rayleigh quotient of $v \in H^1(\Omega)$. Sometimes we write $R_\Omega[v]$ to emphasize the domain over which we take the Rayleigh quotient.

For a triangular domain with side lengths $l_1 \ge l_2 \ge l_3 > 0$, we denote:
\begin{itemize}
  \item[] $D=l_1 =$ diameter,
  \item[] $L=l_1 + l_2 + l_3 =$ perimeter,
  \item[] $A=$ area.
  \end{itemize}
Write $j_{0,1} \simeq 2.4048$ and $j_{1,1} \simeq 3.8317$ for the first positive roots of the Bessel functions $J_0$ and $J_1$, respectively.

\section{\bf Results}

First we develop symmetry and antisymmetry properties of
the fundamental mode of an isosceles triangle. 
\begin{definition}\rm
The \textbf{aperture} of an isosceles triangle is the angle between
its two equal sides. Call a triangle \textbf{subequilateral} if it is isosceles with
aperture less than $\pi/3$, and \textbf{superequilateral} if it is
isosceles with aperture greater than $\pi/3$.
\end{definition}
\begin{theorem} \label{sharpsymmetric} Every fundamental mode of a subequilateral triangle is symmetric
with respect to the line of symmetry of the triangle.
\end{theorem}
See Figure~\ref{symmfig}(a), where the nodal curve is sketched. The theorem is plausible because the main variation of a fundamental mode should take place in the ``long'' direction of the triangle.

We will use this symmetry result when proving the lower bound on the fundamental tone, in Theorem~\ref{th:tld2} below. 

Next we state an antisymmetry result for superequilateral triangles.
\begin{theorem} \label{bluntantisymmetric} Every fundamental mode of a superequilateral triangle is
antisymmetric with respect to the line of symmetry of the triangle.
\end{theorem}
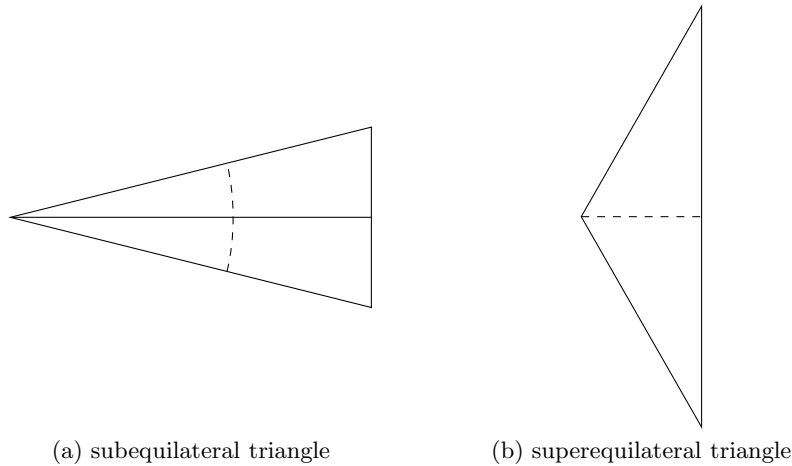
\begin{figure}[t]
  \begin{center}
    \hspace*{\fill}
    \subfloat[subequilateral triangle]{
\begin{tikzpicture}[scale=0.4]
      \path (0,-7) -- (12,7);
      \draw (0,0) -- (12,-3) -- (12,3) -- cycle;
      \clip (0,0) -- (12,-3) -- (12,3) -- cycle;
      \draw (0,0) -- (12,0);
      \draw[dashed] (7.2,-3*7.2/12) .. controls +(3/12*1.1,1.1) and +(3/12*1.1,-1.1) .. (7.2,3*7.2/12);
    \end{tikzpicture}
    }
    \hspace*{\fill}
    \subfloat[superequilateral triangle]{
    \begin{tikzpicture}[scale=0.4]
      \path (0,-7) -- (12,7);
      \draw (4,0) -- (8,7) -- (8,-7) -- cycle;
      \draw[dashed] (4,0) -- (8,0);
    \end{tikzpicture}
    }
    \hspace*{\fill}
  \end{center}
  \caption{Nodal curves (dashed) for the fundamental mode of an isosceles triangle. The fundamental mode satisfies a Neumann condition on each solid line, and a Dirichlet condition on each dashed curve.} \label{symmfig}
\end{figure}

\medskip Now we develop lower bounds on the fundamental tone, under diameter normalization.
The sharp lower bound of Payne and Weinberger \cite{PW60} says that for
convex domains in all dimensions,
\begin{equation} \label{PW}
  \mu_1 D^2>\pi^2.
\end{equation}
The bound is asymptotically correct for rectangular boxes that
degenerate to an interval. 

For triangles we will prove a better lower bound:
\begin{theorem} \label{th:tld2}
For all triangles,
\[
    \mu_1 D^2 > j_{1,1}^2
\]
with equality asymptotically for degenerate acute isosceles
triangles.
\end{theorem}
The theorem improves considerably (for triangles) on Payne and
Weinberger's inequality, because $j_{1,1}^2 \simeq 14.7$ is
greater than $\pi^2 \simeq 9.9$.

\begin{corollary}[Optimal Poincar\'{e} inequality for triangular
domains] For all triangles $T$, one has
\[
\frac{\int_T |\nabla v|^2 \, dA}{\int_T v^2 \, dA} >
\frac{j_{1,1}^2}{D^2}
\]
whenever $v \in H^1(T)$ has mean value zero.
\end{corollary}
The corollary follows immediately, by the Rayleigh characterization of the fundamental tone.

\subsubsection*{Remarks on the literature.} Payne and Weinberger's inequality \eqref{PW} has been generalized to geodesically convex domains on surfaces with nonnegative Gaussian curvature by Chavel and Feldman \cite{CF77}, who adapted Payne and Weinberger's idea of slicing the domain into thin strips. A different approach is to employ a ``$P$-function'' and the maximum principle \cite[Theorem 8.13]{S81}. This approach extends to manifolds of any dimension. It yields only $\mu_1 D^2 \geq \pi^2/4$, but the unwanted factor of $4$ disappears when the fundamental mode is known to have maximum and minimum values of opposite sign and equal magnitude. We do not know whether our inequality for triangles in Theorem~\ref{th:tld2} can be proved by a $P$-function method. 

\medskip Next we deduce lower bounds in terms of perimeter $L$. Since $L>2D$, the Payne--Weinberger lower bound \eqref{PW} implies for all convex, bounded plane domains that
\begin{equation} \label{eq:pwl}
  \mu_1 L^2>4\pi^2 ,
\end{equation}
with equality holding asymptotically for rectangles that degenerate to a segment. We
deduce a stronger inequality for triangles from
Theorem~\ref{th:tld2}.
\begin{corollary} \label{co:tll}
For all triangles,
\[
    \mu_1 L^2>4j_{1,1}^2
\]
with equality asymptotically for degenerate acute isosceles
triangles.
\end{corollary}
The constant $4j_{1,1}^2 \simeq 58.7$ for triangles exceeds the
value $4\pi^2 \simeq 39.5$ for general convex domains in \eqref{eq:pwl}.

Incidentally, the area cannot provide a lower bound on the Neumann
fundamental tone, because for a sequence of triangles degenerating
to a line segment, one finds $\mu_1$ is bounded while the area $A$
approaches zero; thus $\mu_1 A$ can be arbitrarily close to $0$.

\medskip Lastly we examine \textit{upper} bounds on the fundamental tone in terms of diameter. (For upper bounds in terms of area and perimeter, see our paper \cite{LS09a}.) Cheng \cite[Theorem~2.1]{cheng} gave an upper bound for general convex domains
that complements Payne and Weinberger's lower bound; it says in two
dimensions that
\begin{equation} \label{Cheng}
  \mu_1 D^2 < 4j_{0,1}^2 \simeq 23.1.
\end{equation}
A slightly more general result was proved by Ba\~nuelos and Burdzy
\cite[Proposition 2.2]{BaBu} using probabilistic methods. See also
the non-sharp inequality proved using different methods by Smits
\cite[Theorem~4]{S96}.

Our contribution is to obtain a complementary lower bound for all
isosceles triangles of aperture greater than $\pi/3$.
\begin{proposition}\label{1D}
For superequilateral triangles of aperture $\alpha \in (\pi/3,\pi)$, one has
\[
    4j_{0,1}^2 \sin^2(\alpha/2) \leq \mu_1 D^2< 4j_{0,1}^2.
\]
\end{proposition}
Letting $\alpha \to \pi$ yields equality asymptotically in Proposition~\ref{1D}, for degenerate obtuse isosceles triangles. Thus Cheng's upper bound \eqref{Cheng} is best possible even in the restricted class of triangular domains, a fact that has been observed previously in the literature \cite[p.~10]{BaBu}.

Our lower bound in Proposition~\ref{1D} can be improved to $2j_{0,1}^2 (\pi-\alpha) \tan(\alpha/2)$ by combining antisymmetry of the fundamental mode (Theorem~\ref{bluntantisymmetric}) with a sectorial rearrangement result \cite[p.~114]{B79}. The improvement is substantial when the aperture $\alpha$ is close to the equilateral value $\pi/3$. On the other hand, sectorial rearrangement is nontrivial to prove, whereas the lower bound in Proposition~\ref{1D} uses only antisymmetry and domain monotonicity.

\section{\bf The equilateral triangle and its eigenfunctions} \label{equilateral}

This section gathers together the first three Neumann eigenfunctions
and eigenvalues of the equilateral triangle, which we use later to
construct trial functions for close-to-equilateral triangles.

The modes and frequencies of the equilateral triangle were derived
two centuries ago by Lam\'{e}, albeit without a proof of
completeness. We present the first few modes below. For proofs,
see the recent exposition (including completeness) by McCartin
\cite{M02}, building on work of Pr\'{a}ger \cite{Pr98}. A
different approach is due to Pinsky \cite{Pi80}.

Consider the the equilateral triangle $E$ with vertices at $(0,0)$,
$(1,0)$ and $(1/2,\sqrt{3}/2)$. Then $\mu_0=0$, with eigenfunction
$u_0 \equiv 1$, and
\[
\mu_1 = \mu_2 = \frac{16 \pi^2}{9}
\]
with eigenfunctions
  \begin{align*}
    u_1(x,y) & = 2 \Big[ \cos \big( \frac{\pi}{3}(2x-1) \big) + \cos \big( \frac{2\pi y}{\sqrt{3}} \big) \Big] \sin \big( \frac{\pi}{3} (2x-1)
    \big) , \\
    u_2(x,y) & = \cos \big( \frac{2\pi}3(2x-1) \big) - 2\cos\big( \frac\pi3(2x-1) \big) \cos\big( \frac{2\pi y}{\sqrt{3}} \big) .
  \end{align*}
Clearly $u_1$ is antisymmetric with respect to the line of symmetry
$\{ x=1/2 \}$ of the equilateral triangle, since
$u_1(1-x,y)=-u_1(x,y)$, whereas $u_2$ is symmetric with respect to
that line.

We evaluate some integrals of $u_1$ and $u_2$, for later use:
\begin{align*}
\int_E u_1^2 \, dA & = \int_E u_2^2 \, dA = \frac{3\sqrt{3}}{8} , \\
\int_E \Big( \frac{\partial u_1}{\partial x} \Big)^{\! 2} \, dA & = \int_E \Big( \frac{\partial u_2}{\partial y} \Big)^{\! 2} \, dA = \frac{32\pi^2 + 243}{32\sqrt{3}} , \\
\int_E \Big( \frac{\partial u_1}{\partial y} \Big)^{\! 2} \, dA &
= \int_E \Big( \frac{\partial u_2}{\partial x} \Big)^{\! 2} \, dA
= \frac{32\pi^2 - 243}{32\sqrt{3}} .
\end{align*}

\section{\bf Isosceles triangles}

In this section we focus on isosceles triangles, establishing bounds that will help show symmetry
of the fundamental mode for subequilateral triangles (in
Section~\ref{isec3}) and antisymmetry of the fundamental mode for
superequilateral triangles (in Section~\ref{isec4}).

\subsection{Bounds for sub- and super-equilateral triangles} \label{isec1}

First we bound the fundamental tone of an isosceles triangle, by
transplanting it to a sector. Write the polar coordinates as $(r,\theta)$, let $l>0$, and define
\[
S(\alpha) = \{ (x,y) : 0<r<l, |\theta| < \alpha/2 \}
\]
to be the sector of aperture $0<\alpha<2\pi$ and side length $l$.
\begin{lemma} \label{le:sector}
When $\alpha < \pi/2.68$, the sector $S(\alpha)$ has
fundamental tone
\[
\mu_1 \big( S(\alpha) \big) = (j_{1,1}/l)^2
\]
and fundamental mode $J_0(j_{1,1}r/l)$.
\end{lemma}
This fundamental mode is symmetric with respect to the $x$-axis, which is the line of symmetry of the sector.
\begin{proof}
Consider an aperture $\alpha \in (0,2\pi)$. Fix the side length to be $l=1$, by rescaling. 

By a standard separation of variables argument, one reduces to comparing the
following two eigenvalues. First, one has the eigenvalue
$j_{1,1}^2$ associated with the nonconstant radial mode
$J_0(j_{1,1} r)$, where the Neumann boundary condition is satisfied at $r=1$ because $J_0^\prime = - J_1$. Second, one has the eigenvalue
$(j_{\nu,1}^\prime)^2$ associated with the angular mode
$J_\nu(j_{\nu,1}^\prime r) \sin (\nu \theta)$,  where $\nu=\pi/\alpha$. This angular mode is antisymmetric with respect to the $x$-axis.

We will show the radial mode gives the lower eigenvalue, when the
aperture is less than $\pi/2.68$, that is, when $\nu > 2.68$. 

It is known that $j_{\nu,1}^\prime$ is a strictly increasing function of $\nu$; one can consult the original proof in \cite[p.~510]{W52}, or the more elementary proof in \cite{L90}. Since $j_{2.68,1}^\prime \simeq 3.8384 > 3.8317 \simeq j_{1,1}$, we conclude $j_{\nu,1}^\prime > j_{1,1}$ when $\nu > 2.68$, which proves the lemma.

Incidentally, more precise numerical work reveals that the transition occurs at $\nu=2.6741$, to four decimal places.
\end{proof}

Next consider the isosceles triangle $T(\alpha)$ having aperture
$0<\alpha<\pi$, equal sides of length $l$, and vertex at the origin.
After rotating the triangle to make it symmetric about the
positive $x$-axis, it can be written as
\[
T(\alpha) = \{ (x,y) : 0<x<l\cos(\alpha/2), |y|<x\tan(\alpha/2) \} .
\]
Write $\mu_1(\alpha)$ for the fundamental tone of $T(\alpha)$. We will bound this tone in terms of the aperture.

\begin{lemma}\label{boundsiso}
When $0 < \alpha < \pi/3$, the subequilateral triangle $T(\alpha)$ satisfies
\[
\frac{j_{1,1}^2}{1+\tan(\alpha/2)+\tan^2(\alpha/2)} < \mu_1(\alpha) D^2 \leq \frac{j_{1,1}^2}{\cos^2(\alpha/2)} .
\]
\end{lemma}

\begin{proof}
Note $T(\alpha)$ has diameter $D=l$, since the triangle is subequilateral.

We will transplant eigenfunctions between the triangle and the sector by stretching in the radial direction, a technique borrowed from work on Dirichlet eigenvalues by Freitas \cite[\S3]{F07}. Define
\[
\rho(\theta) = \frac{\cos(\alpha/2)}{\cos (\theta)} , \qquad |\theta| < \alpha/2 ,
\]
so that the transformation
\[
\sigma(r,\theta) = \big(r \rho(\theta),\theta \big)
\]
maps the sector $S(\alpha)$ onto the isosceles triangle
$T(\alpha)$.

\smallskip
(a) For the lower bound in the lemma, let $v$ be a fundamental mode of the triangle. We use $(v+C)\circ \sigma$ as a trial function for the sector; the constant $C$ is chosen to ensure the trial function has mean value zero, $\int_{S(\alpha)} (v+C) \circ \sigma \, dA =0$. Then the denominator of the Rayleigh quotient is
\begin{align*}
\int_{S(\alpha)} |(v+C)\circ \sigma|^2 \, dA
& = \int_{T(\alpha)} (v+C)^2 \frac{r}{\rho} \, \frac{dr}{\rho}d\theta && \text{by $r \mapsto r/\rho(\theta)$} \\
& > \int_{T(\alpha)} (v+C)^2 \, rdrd\theta && \text{since $\rho < 1$} \\
& \geq \int_{T(\alpha)} v^2 \,dA
\end{align*}
since $\int_{T(\alpha)} v \, dA = 0$ and $C^2 \geq 0$.

For the numerator of the Rayleigh quotient, we first apply the chain rule:
\[
(v \circ \sigma)_r = (v_r \circ \sigma) \cdot \rho , \qquad (v \circ \sigma)_\theta = (v_r \circ \sigma)\cdot r \rho^\prime + (v_\theta \circ \sigma) .
\]
Hence the numerator is
\begin{align}
& \int_{S(\alpha)} |\nabla ( (v+C)\circ \sigma)|^2 \, dA \notag \\
& = \int_{S(\alpha)} \left[ (v_r \circ \sigma)^2 \rho^2 + \frac{1}{r^2}\left( (v_r \circ \sigma)\cdot r \rho^\prime + (v_\theta \circ \sigma) \right)^2 \right] \, r drd\theta \notag \\
& = \int_{T(\alpha)} \left[ v_r^2+\frac{1}{r^2}\Big( v_r \frac{r}{\rho} \rho^\prime + v_\theta \Big)^{\! 2} \right] \, r drd\theta \qquad \text{by $r \mapsto r/\rho(\theta)$} \label{eq:varchange} \\
& \leq \int_{T(\alpha)} \left[ v_r^2\left(1+\left(\frac{\rho^\prime}{\rho}\right)^{\! 2}+\left|\frac{\rho^\prime}{\rho}\right|\right)+\frac{v_\theta^2}{r^2}\left(1+\left|\frac{\rho^\prime}{\rho}\right| \right)\right] dA \notag \\
& \leq [1+\tan(\alpha/2)+\tan^2(\alpha/2)] \int_{T(\alpha)} |\nabla v|^2 \, dA \notag
\end{align}
since $|\rho^\prime/\rho|=|\tan(\theta)| \leq \tan(\alpha/2)$. Combining the numerator and denominator, we see
\[
  \mu_1 \big( S(\alpha) \big) \leq \frac{\int_{S(\alpha)} |\nabla ( (v+C)\circ \sigma)|^2 \, dA}{\int_{S(\alpha)} |(v+C)\circ \sigma|^2 \, dA} < [1+\tan(\alpha/2)+\tan^2(\alpha/2)] \mu_1(\alpha).
\]
Recalling that $\mu_1 \big( S(\alpha) \big) = (j_{1,1}/l)^2$ and $D=l$, we deduce the lower bound in Lemma~\ref{boundsiso}.

\smallskip
(b) To get an upper bound, we transplant the eigenfunction of the sector to yield a trial
function for the triangle. Write
\[
  v(r)=J_0 \left( j_{1,1}\frac{r}{l} \right)
\]
for the fundamental mode of the sector $S(\alpha)$. Notice $v$ is also a radial mode of the disk $\D(l)$ of radius $l$, satisfying $-\Delta v = (j_{1,1}/l)^2 v$ with normal derivative $\partial v/\partial r = 0$ at $r=l$ (using that $J_0^\prime = -J_1$). Hence $\int_{\D(l)} v \, dA = 0$ by the divergence theorem. Thus the transplanted eigenfunction $v \circ \sigma^{-1}$ integrates to $0$ over the triangle $T(\alpha)$:
\begin{align*}
\int_{T(\alpha)} (v \circ \sigma^{-1}) \, dA
& = \int_{S(\alpha)} v \, \rho^2 rdr d\theta \qquad \text{since $\sigma^{-1}(r,\theta)=(r/\rho(\theta),\theta)$} \\
& = \int_{-\alpha/2}^{\alpha/2} \rho^2 \, d\theta \int_0^l v \, rdr \\
& = \frac{1}{2\pi} \big( 2\sin (\alpha/2) \cos (\alpha/2) \big) \int_{\D(l)} v \, dA \\
& = 0 .
\end{align*}
Similarly, the denominator of the Rayleigh quotient for $v \circ \sigma^{-1}$ evaluates to
\[
\int_{T(\alpha)} (v \circ \sigma^{-1})^2 \, dA
= \frac{1}{2\pi} (2\sin \alpha/2 \cos \alpha/2) \int_{\D(l)} v^2 \, dA .
\]
The numerator of the Rayleigh quotient equals
\begin{align*}
& \int_{T(\alpha)} |\nabla (v \circ \sigma^{-1})|^2 \, dA \\
& = \int_{S(\alpha)} \left[ v_r^2+\frac{1}{r^2}\left(v_r r\rho (1/\rho)^\prime + v_\theta\right)^{\! 2} \right] \, r drd\theta && \text{by arguing like for \eqref{eq:varchange}} \\
& = \int_{-\alpha/2}^{\alpha/2} \big( 1 + (\rho^\prime/\rho)^2 \big) \, d\theta \int_0^l v_r^2 \, rdr && \text{since $v_\theta \equiv 0$} \\
& = \frac{1}{2\pi} \big( 2 \tan (\alpha/2) \big) \int_{\D(l)} |\nabla v|^2 \, dA .
\end{align*}
We conclude
\begin{align*}
\mu_1(\alpha)
& \leq R[v \circ \sigma^{-1}] \\
& = \frac{2 \tan (\alpha/2)}{2\sin (\alpha/2) \cos (\alpha/2)} \, \frac{\int_{\D(l)} |\nabla v|^2 \, dA}{\int_{\D(l)} v^2 \, dA} \\
& = \frac{1}{\cos^2 (\alpha/2)} (j_{1,1}/l)^2 ,
\end{align*}
which gives the upper estimate in Lemma~\ref{boundsiso}.
\end{proof}

Next we give a proof for superequilateral triangles of Cheng's bound \eqref{Cheng}.

\begin{lemma} \label{chengsupereq}
When $\pi/3 < \alpha < \pi$, the superequilateral triangle $T(\alpha)$ satisfies
\[
\mu_1(\alpha) D^2 < 4 j_{0,1}^2 .
\]
Here the diameter is $D=2 l \sin(\alpha/2)$, which is the length of the vertical side.
\end{lemma}

\begin{proof}
We simply specialize Cheng's method to triangles. Denote the upper and lower vertices of $T(\alpha)$ by
\[
z_\pm = \big( l \cos(\alpha/2) , \pm l \sin(\alpha/2) \big) ,
\]
so that the diameter is $D=|z_+ - z_-|$. Write
\[
v_0(z)=J_0 \! \left( j_{0,1}\frac{|z|}{D/2} \right)
\]
for the fundamental mode of the disk of radius $D/2$. Define a trial function for $z=(x,y)$ in $T(\alpha)$ by
\[
v(z) =
\begin{cases}
+v_0(z-z_+) , & \text{if $|z-z_+|<D/2$,} \\
-v_0(z-z_-) , & \text{if $|z-z_-|<D/2$,} \\
0 , & \text{otherwise.}
\end{cases}
\]
Notice $v$ is continuous and piecewise smooth with mean value zero, and is supported on two circular sectors. Since $v_0$ is radial, we compute
\[
\mu_1(\alpha) \leq R_{T(\alpha)}[v] = R_{\D(D/2)}[v_0] = \Big( \frac{j_{0,1}}{D/2} \Big)^{\! 2} .
\]
Equality cannot hold, since $v$ is not smooth and hence is not an eigenfunction for $T(\alpha)$.
\end{proof}

For general convex domains, the support of $v$ is not just a union of sectors, and thus further arguments are required to prove Cheng's general bound.

\subsection{Linear transformations for isosceles triangles} \label{isec2}

Define a diagonal linear transformation
\[
    \tau(x,y) = \Big( x\frac{\cos(\beta/2)}{\cos(\alpha/2)}, y\frac{\sin(\beta/2)}{\sin(\alpha/2)} \Big)
\]
mapping $T(\alpha)$ onto $T(\beta)$. We develop a result about transplanting eigenfunctions from one isosceles triangle to another. (The method applies to more general domains, such as ellipses, whenever the domains map to one another under a diagonal linear transformation.)

\begin{lemma}\label{lemcomp}
Let $\mu(\alpha)$ and $\mu(\beta)$ be eigenvalues of the triangles $T(\alpha)$ and $T(\beta)$ respectively, for some $\alpha, \beta \in (0,\pi)$. Let $w$ be a nonconstant eigenfunction belonging to $\mu(\beta)$, and assume $w\circ\tau$  can be used as a trial function for $\mu(\alpha)$, meaning $\mu(\alpha) \leq R[w \circ \tau]$.

If either condition (i) or (ii) below holds, for some real number $G(\beta)$, then
\[
     \mu(\alpha) < \big( 1+G(\beta) \big) \mu(\beta) .
\]
The conditions are:

(i) $\alpha<\beta$ and
\[
    \frac{\int_{T(\beta)} w_y^2 \, dA}{\int_{T(\beta)} (w_x^2+w_y^2) \, dA}
    < \sin^2(\alpha/2)+\frac{\sin^2(\alpha/2)\cos^2(\alpha/2)}{\sin^2(\beta/2)-\sin^2(\alpha/2))}G(\beta) ;
\]

(ii) $\alpha>\beta$ and
\[
    \frac{\int_{T(\beta)} w_y^2 \, dA}{\int_{T(\beta)} (w_x^2+w_y^2) \, dA}
    > \sin^2(\alpha/2)+\frac{\sin^2(\alpha/2)\cos^2(\alpha/2)}{\sin^2(\beta/2)-\sin^2(\alpha/2))}G(\beta) .
\]
\end{lemma}

\begin{proof}
Observe
\[
\frac{\sin^2(\beta/2)}{\sin^2(\alpha/2)} - \frac{\cos^2(\beta/2)}{\cos^2(\alpha/2)} = \frac{\sin^2(\beta/2) - \sin^2(\alpha/2)}{\sin^2(\alpha/2) \cos^2(\alpha/2)} .
\]
Multiplying these expressions on the left and right of (i), respectively, implies that
\begin{equation} \label{eq:G}
\frac{\cos^2(\beta/2)}{\cos^2(\alpha/2)} (1-\kappa ) + \frac{\sin^2(\beta/2)}{\sin^2(\alpha/2)} \kappa < 1+G(\beta)
\end{equation}
where $\kappa = \int_{T(\beta)} w_y^2 \, dA \Big/ \int_{T(\beta)} (w_x^2+w_y^2) \, dA$. The same holds for (ii). Therefore
\begin{align*}
\mu(\alpha) & \leq R[w \circ \tau] \\
& = \frac{\int_{T(\alpha)}\left(\frac{\cos^2(\beta/2)}{\cos^2(\alpha/2)}w_x^2+\frac{\sin^2(\beta/2)}{\sin^2(\alpha/2)}w_y^2\right)\circ \tau \, dA}{\int_{T(\alpha)}|w\circ \tau|^2 \, dA} \\
& = \frac{\int_{T(\beta)} \left(\frac{\cos^2(\beta/2)}{\cos^2(\alpha/2)}w_x^2+\frac{\sin^2(\beta/2)}{\sin^2(\alpha/2)}w_y^2\right) \, dA}{\int_{T(\beta)} w^2 \, dA} \qquad \text{by changing variable} \\
& = \left( \frac{\cos^2(\beta/2)}{\cos^2(\alpha/2)} (1-\kappa) + \frac{\sin^2(\beta/2)}{\sin^2(\alpha/2)} \kappa \right) R[w] \\
& < \big( 1+G(\beta) \big) R[w] \label{mucomp}
\end{align*}
by \eqref{eq:G}. Since $w$ is an eigenfunction for $\mu(\beta)$ we have $R[w]=\mu(\beta)$, and so the proof is complete.
\end{proof}

The lemma simplifies considerably when the number $G(\beta)$ is zero:
\begin{corollary}\label{corcomp}
Let $\mu(\alpha)$ and $\mu(\beta)$ be eigenvalues of the triangles $T(\alpha)$ and $T(\beta)$ respectively, for some $\alpha, \beta \in (0,\pi)$. Let $w$ be a nonconstant eigenfunction belonging to $\mu(\beta)$, and assume $w\circ\tau$  can be used as a trial function for $\mu(\alpha)$, meaning $\mu(\alpha) \leq R[w \circ \tau]$. Then $\mu(\alpha) < \mu(\beta)$ if
\begin{align*}
\text{(i)} \qquad \alpha<\beta & \quad \text{and} \quad \frac{\int_{T(\beta)} w_y^2 \, dA}{\int_{T(\beta)} w_x^2 \, dA} < \tan^2(\alpha/2) \\
\text{or \ (ii)} \qquad \alpha>\beta & \quad \text{and} \quad \frac{\int_{T(\beta)} w_y^2 \, dA}{\int_{T(\beta)} w_x^2 \, dA} > \tan^2(\alpha/2) .
\end{align*}
\end{corollary}

\section{\bf Proof of Theorem~\ref{sharpsymmetric}: symmetry of the fundamental mode for subequilateral triangles} \label{isec3}

Eigenfunctions of an isosceles triangle can be assumed either symmetric or antisymmetric. Indeed, any eigenfunction $v$ can be decomposed into the sum of its symmetric part $(v+v^r)/2$ and antisymmetric part $(v-v^r)/2$, where $v^r$ denotes the reflection of $v$ across the line of symmetry of the triangle. Each of these two parts is itself an eigenfunction, unless it is identically zero (as happens when the eigenfunction is already symmetric or antisymmetric).

We will show that for a subequilateral triangle, the fundamental tone satisfies $\mu_1 D^2 < 16\pi^2/9$, whereas the smallest eigenvalue $\mu_a$ having an antisymmetric eigenfunction satisfies $\mu_a D^2 > 16\pi^2/9$. (See Figure~\ref{mustarfig}.) It follows that every fundamental mode of the triangle is symmetric.

\begin{figure}[t]
  \begin{center}
\begin{tikzpicture}[smooth,xscale=5,yscale=0.4]
  \draw[<->] (1.2,12) node [below] {\tiny $\alpha$} -- (0,12) -- (0,25);
      \clip (-0.3,11) rectangle (1.2,25);
      \draw (0,12) -- +(0,-0.12) node [below] {\tiny $0$};
      \draw (1.0472,12) -- +(0,-0.12) node [below] {\tiny $\pi/3$};
      \draw[loosely dotted] (pi/3,12) -- +(0,16*pi^2/9-12);
      \draw (0,12) -- +(-0.02,0) node [left] {\tiny $12$};
      \draw (0,16) -- +(-0.02,0) node [left] {\tiny $16$};
      \draw (0,20) -- +(-0.02,0) node [left] {\tiny $20$};
      \draw (0,24) -- +(-0.02,0) node [left] {\tiny $24$};
      \draw (0,14.682) -- +(-0.02,0) node [left] {\tiny $j_{1,1}^2$};
      \draw (0,17.546) -- +(-0.02,0) node [left] {\tiny $16\pi^2/9$};
      \draw[loosely dotted] (0,17.546) -- +(pi/3,0);
      \draw[densely dotted,domain=0:0.33] plot (\x,{3.8317^2/(1+tan(\x*90/3.1416)^2+tan(\x*90/3.1416))});
      \draw[densely dotted,domain=0:0.9] plot (\x,{3.8317^2/(cos(\x*90/3.1416)^2)});
      \draw plot coordinates {
( 1.0472,17.5460)( 1.0001,17.2927)( 0.9529,17.0484)(
0.9058,16.8145)( 0.8587,16.5920)( 0.8116,16.3816)( 0.7645,16.1837)(
0.7173,15.9985)( 0.6702,15.8261)( 0.6231,15.6665)( 0.5760,15.5195)(
0.5288,15.3850)( 0.4817,15.2629)( 0.4346,15.1530)( 0.3875,15.0551)(
0.3403,14.9689)( 0.2932,14.8943)( 0.2461,14.8312)( 0.1990,14.7793)(
0.1518,14.7386)( 0.1047,14.7089)(0.05,14.688)(0,14.682) }; \path
(0.7645,16.1837) node [below=-1pt] {\tiny $\mu_1 D^2$};
\draw[densely dashed] plot coordinates {
( 1.0472,17.5460)( 1.0001,18.8398)( 0.9529,20.3180)( 0.9058,22.0180)( 0.8587,23.9871)( 0.8116,26.2863)( 0.7645,28.9953)( 0.7173,32.2201)
}; \path ( 0.9529,20.3180) node [left=-1pt] {\tiny $\mu_a D^2$};
    \end{tikzpicture}
  \end{center}
  \caption{Numerical plot of the smallest Neumann eigenvalue $\mu_1$ and the smallest Neumann eigenvalue $\mu_a$ with antisymmetric eigenfunction, for subequilateral triangles with aperture $\alpha$. The eigenvalues are normalized by  multiplying by the square of the diameter. Dotted lines show the bounds from Lemma \ref{boundsiso}, converging to the asymptotic value $j_{1,1}^2$.} \label{mustarfig}
\end{figure}

Take the equal sides of the isosceles triangle $T(\alpha)$ to have length $l=1$. Assume $\alpha<\pi/3$ so that $T(\alpha)$ is subequilateral with diameter $D=1$, and take $\beta=\pi/3$ so that $T(\beta)$ is equilateral. Let $w$ be the eigenfunction of the equilateral triangle $T(\beta)$ that is symmetric with respect to the $x$-axis (meaning $w$ is obtained from the eigenfunction $u_2$ in Section~\ref{equilateral} by first translating $E$ to shift its vertex $(1/2,\sqrt{3}/2)$ to the origin and then rotating by $\pi/2$ counterclockwise).

Recall the linear transformation $\tau$ from Section~\ref{isec2}, which maps $T(\alpha)$ onto $T(\beta)$. Note $w \circ \tau$ has mean value zero over $T(\alpha)$, and hence can be used as a trial function for the fundamental mode $\mu_1(\alpha)$. Condition (i) in Corollary~\ref{corcomp} is equivalent to
\[
  0.130 \simeq \frac{32\pi^2-243}{32\pi^2+243} = \frac{\int_E u_{2,x}^2 \, dA}{\int_E u_{2,y}^2 \, dA} < \tan^2(\alpha/2) ,
\]
where the integrals of $u_2$ were evaluated in
Section~\ref{equilateral}. Therefore by
Corollary~\ref{corcomp}(i),
\begin{equation} \label{mustarup}
  \mu_1(\alpha) < \mu_1(\pi/3)=\frac{16\pi^2}{9}
\end{equation}
if $\tan^2(\alpha/2) \gtrsim 0.130$. On the other hand, the upper bound from Lemma~\ref{boundsiso} gives \eqref{mustarup} whenever
\[
  \cos^2(\alpha/2) > \frac{9j_{1,1}^2}{16\pi^2}  ,
\]
which is equivalent to $\tan^2(\alpha/2) < (16\pi^2/9j_{1,1}^2)-1 \simeq 0.195$.
Thus \eqref{mustarup} holds for all $\alpha<\pi/3$.

Note that our proof of \eqref{mustarup} relies on transplanting
trial functions from both the equilateral triangle (for ``large''
$\alpha$, near $\pi/3$) and the sector (for ``small'' $\alpha$,
when we call on Lemma~\ref{boundsiso}).

Now change notation and take $\alpha=\pi/3$ and $\beta<\pi/3$. Consider the smallest eigenvalue of $T(\beta)$ that has an antisymmetric eigenfunction; call this eigenvalue $\mu_a(\beta)$ and its corresponding antisymmetric eigenfunction $v$. Note $v \circ \tau$ has mean value zero, and hence can be used as a trial function for the fundamental tone $\mu_1(\alpha)$ of the equilateral triangle. By  Corollary~\ref{corcomp}(ii) we see
\begin{equation} \label{eq:antisym}
  \mu_a(\beta) > \mu_1(\pi/3) = \frac{16\pi^2}{9}
\end{equation}
if
\[
  \frac{\int_{T(\beta)} v_y^2 \, dA}{\int_{T(\beta)} v_x^2 \, dA}
  > \tan^2(\pi/6)=\frac{1}{3} .
\]
Assume, on the other hand, that this last condition does not hold. Then $\int_{T(\beta)} v_x^2 \, dA \geq 3 \int_{T(\beta)} v_y^2 \, dA$. Write $h=\cos(\beta/2)$ for the width of $T(\beta)$ and $\gamma=\tan(\beta/2)$ for the slope of its upper side. Then
\[
\mu_a(\beta) = \frac{\int_{T(\beta)} |\nabla v|^2 \, dA}{\int_{T(\beta)} v^2 \, dA}
\geq \frac{4 \int_0^h \int_{-\gamma x}^{\gamma x} v_y^2 \, dy dx}{\int_0^h \int_{-\gamma x}^{\gamma x} v^2 \, dy dx} .
\]
Notice that for each fixed $x$, the function $y \mapsto v(x,y)$ has mean value zero, by the antisymmetry. Hence $y \mapsto v(x,y)$ is a valid trial function for the fundamental tone of the one dimensional Neumann problem on the interval $[-\gamma x, \gamma x]$. That fundamental tone equals $(\pi/2\gamma x)^2$, and so
\[
\frac{\int_{-\gamma x}^{\gamma x} v_y^2 \, dy}{\int_{-\gamma x}^{\gamma x} v^2 \, dy} \geq \Big( \frac{\pi}{2\gamma x} \Big)^{\! 2} .
\]
Since $x \leq h$, we conclude from above that
\[
\mu_a(\beta) \geq 4 \Big( \frac{\pi}{2\gamma h} \Big)^{\! 2} = \frac{\pi^2}{\sin^2(\beta/2)} .
\]
Then because $\beta<\pi/3$ we deduce $\mu_a(\beta) > 4\pi^2$, which is certainly greater than $16\pi^2/9$. Hence \eqref{eq:antisym} holds for all $\beta<\pi/3$.

We have shown that
\[
\mu_1(\alpha) < \frac{16\pi^2}{9} < \mu_a(\beta)
\]
whenever $\alpha,\beta < \pi/3$. The proof is complete.

\subsection*{Method of the Unknown Trial Function} Our proof above uses the
antisymmetric eigenfunction to construct antisymmetric trial
functions for the two ``endpoint'' situations: the equilateral
triangle and the narrow isosceles triangle (via the interval in
the $y$-direction, above). We know those two endpoint eigenvalues
exactly, and so we obtain the lower bound \eqref{eq:antisym} on
the antisymmetric eigenvalue.

We call this approach the ``Method of the Unknown Trial
Function'', because we do not know the antisymmetric eigenfunction
explicitly, and for a given aperture $\beta$ we do not even know
\textit{which} of the two endpoint situations will give the
lower bound \eqref{eq:antisym}.

The method will be used again in the proof of
Theorem~\ref{bluntantisymmetric}, using different endpoint cases.

\section{\bf Proof of Theorem~\ref{bluntantisymmetric}: antisymmetry of the fundamental mode for superequilateral triangles} \label{isec4}

We will show that for a superequilateral triangle, the fundamental tone $\mu_1$ is smaller than the smallest eigenvalue $\mu_s$ having a symmetric eigenfunction. See Figure~\ref{fig:supereq}. It follows that every fundamental mode of the triangle is antisymmetric.

\begin{figure}[t]
  \begin{center}
\begin{tikzpicture}[smooth,xscale=3,yscale=0.1]
  \draw[<->] (3.3,16) node [below] {\tiny $\beta$} -- (pi/3,16) -- (pi/3,62);
      \draw (pi/2,16) -- +(0,-0.5) node [below] {\tiny $\pi/2$};
      \draw (pi/3,16) -- +(0,-0.5) node [below] {\tiny $\pi/3$};
      \draw (pi,16) -- +(0,-0.5) node [below] {\tiny $\pi$};
      \draw[loosely dotted] (pi/3,23.18) -- (pi,23.18);
      \draw[loosely dotted] (pi/3,58.727) -| (pi,16);
      \draw[loosely dotted] (pi/2,16) |- (pi/3,4*pi^2);
      \draw (pi/3,4*pi^2) -- +(-0.02,0) node [left] {\tiny $4\pi^2$};
      \draw (pi/3,58.727) -- +(-0.02,0) node [left] {\tiny $4j_{1,1}^2$};
      \draw (pi/3,23.18) -- +(-0.02,0) node [left] {\tiny $4j_{0,1}^2$};
      \draw[densely dotted,domain=2:pi] plot (\x,{23.18*sin(deg(\x/2))^2});
      \draw (pi/3,17.546) -- +(-0.02,0) node [left] {\tiny $16\pi^2/9$};
      \draw[densely dashed] plot coordinates {
(1.0472,17.5460)( 1.0996,19.4768)( 1.1519,21.5119)( 1.2043,23.6435)( 1.2566,25.8595)( 1.3090,28.1419)( 1.3614,30.4660)( 1.4137,32.8003)( 1.4661,35.1073)( 1.5184,37.3464)( 1.5708,39.4785)( 1.7237,44.8402)( 1.8766,48.8091)( 2.0296,51.6706)( 2.1825,53.7772)( 2.3354,55.3675)( 2.4883,56.5804)( 2.6412,57.4933)( 2.7942,58.1511)( 2.9471,58.5981)
} -- (pi,58.7279);
\draw plot coordinates {
( 1.0472,17.5460)( 1.0996,17.7879)( 1.1519,18.0243)( 1.2043,18.2556)( 1.2566,18.4818)( 1.3090,18.7031)( 1.3614,18.9196)( 1.4137,19.1314)( 1.4661,19.3386)( 1.5184,19.5412)( 1.5708,19.7392)
( 1.7237,20.2918)( 1.8766,20.8052)( 2.0296,21.2783)( 2.1825,21.7085)( 2.3354,22.0927)( 2.4883,22.4262)( 2.6412,22.7038)( 2.7942,22.9191)( 2.9471,23.0656)( 3.1000,23.1767)} -- (pi,23.18);
\draw (2.0296,21.2783) node[below=-1pt] {\tiny $\mu_1 D^2$};
\draw (2.0296,51.6706) node[below right=-3pt] {\tiny $\mu_s D^2$};
    \end{tikzpicture}
  \end{center}
  \caption{Numerical plot of the smallest Neumann eigenvalue $\mu_1$ and the smallest Neumann eigenvalue $\mu_s$ having a symmetric eigenfunction, for superequilateral triangles with aperture $\beta$. The eigenvalues are normalized by  multiplying by the square of the diameter. The dotted line shows the bound from Proposition~\ref{1D}, converging to the asymptotic value $4j_{0,1}^2$.} \label{fig:supereq}
\end{figure}
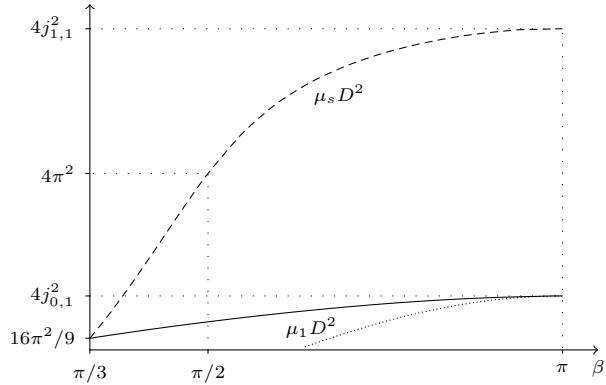 

Our proof will rely on Theorem~\ref{th:tld2}, but there is no
danger of logical circularity because
Theorem~\ref{bluntantisymmetric} plays no role in proving Theorem~\ref{th:tld2}; it is used only in the proof of Proposition~\ref{1D}.

Let us continue to assume that the equal sides of the isosceles triangle
$T(\beta)$ have length $l=1$. Assume $\pi/3 < \beta < \pi$, so that the triangle is superequilateral with diameter $D=2\sin(\beta/2)$.

Let $\mu_s(\beta)$ denote the smallest positive eigenvalue of $T(\beta)$
that has a symmetric eigenfunction $w$. (The nodal domains for this symmetric eigenfunction are sketched in Figure~\ref{fig:symmmode}, based on numerical work.) Cut $T(\beta)$ in half along its line of symmetry, and call the upper half right triangle $U(\beta)$. Then $w$ has mean value zero over each half of $T(\beta)$, and thus is a valid trial function for $\mu_1 \big( U(\beta) \big)$. Hence
\[
\mu_s(\beta) = R_{T(\beta)}[w] = R_{U(\beta)}[w] \geq \mu_1 \big( U(\beta) \big) > j_{1,1}^2
\]
by Theorem \ref{th:tld2}, since $U(\beta)$ has diameter $1$. Further, Cheng's bound \eqref{Cheng} gives an upper bound
\[
  \mu_1(\beta)<\frac{4j_{0,1}^2}{D^2} = \frac{j_{0,1}^2}{\sin^2(\beta/2)}.
\]
Combining the two estimates, we deduce that if
\[
\sin^2(\beta/2) \geq \frac{j_{0,1}^2}{j_{1,1}^2} \simeq 0.39
\]
then $\mu_1(\beta)<\mu_s(\beta)$. In particular, for obtuse isosceles triangles ($\pi/2 \leq \beta < \pi$) we deduce the fundamental mode must be antisymmetric. 

\begin{figure}[t]
    \hspace{\fill}
    \subfloat[\text{acute triangle}]{
    \begin{tikzpicture}[scale=1.5]
      \path (60:2) -- (-60:2);
      \draw[clip] (0,0) -- (37:2) -- (-37:2) -- cycle;
      \draw[dashed] (0,0) circle (1.1);
      \draw (0,0) -- (2,0);
    \end{tikzpicture}
    }
    \hspace{\fill}
    \subfloat[right triangle]{
    \begin{tikzpicture}[scale=1.5]
      \path (60:2) -- (-60:2);
      \draw[clip] (0,0) -- (45:2) -- (-45:2) -- cycle;
      \draw (0,0) -- (2,0);
      \draw[dashed] (45:1) -- ({sqrt(2)},0) -- (-45:1);
    \end{tikzpicture}
    }
    \hspace{\fill}
    \subfloat[obtuse triangle]{
    \begin{tikzpicture}[scale=1.5]
      \path (-0.5,0) -- (1.5,0);
      \draw[clip] (0,0) -- (60:2) -- (-60:2) -- cycle;
      \draw (0,0) -- (2,0);
      \draw[dashed] (60:2) circle (1.15);
      \draw[dashed] (-60:2) circle (1.15);
    \end{tikzpicture}
    }
    \hspace{\fill}
  \caption{Nodal curves (dashed) for the lowest symmetric modes of isosceles triangles. The symmetric mode satisfies a Neumann condition on each solid line, and a Dirichlet condition on each dashed curve.}
\label{fig:symmmode}
\end{figure}
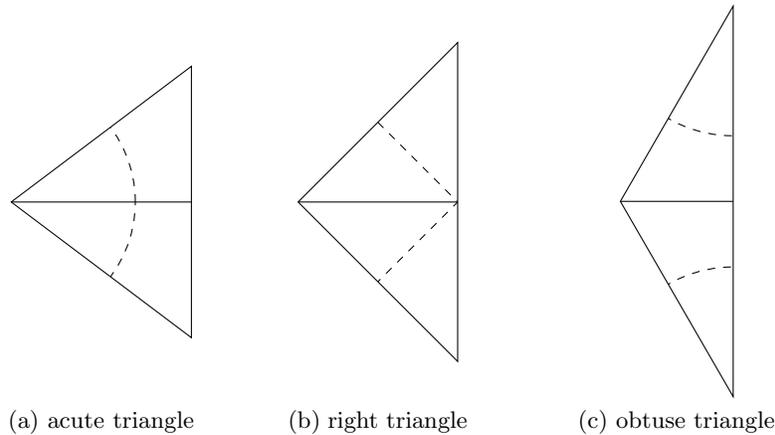

Suppose from now on that $\pi/3 < \beta < \pi/2$. We have
\begin{equation} \label{eq:mu1upper}
  \mu_1(\beta) <
  \frac{16\pi^2}{3S^2}=\frac{16\pi^2}{12\sin^2(\beta/2)+6},
\end{equation}
where $S^2=l_1^2+l_2^2+l_3^2$ is the sum of squares of side lengths of the triangle,
by Theorem~3.1 in our companion paper \cite{LS09a}. 

To show that $\mu_s(\beta)$ exceeds this last value, we employ our
Method of the Unknown Trial Function, this time with the
``unknown'' function being the symmetric eigenfunction $w$, and
with certain equilateral and right triangles providing the
endpoint cases.

Let
\[
\kappa = \frac{\int_{T(\beta)} w_y^2 \, dA}{\int_{T(\beta)} (w_x^2+w_y^2) \, dA} .
\]
The proof will divide into two cases, depending on whether $\kappa < 1/2$ or $\kappa \geq 1/2$.

Take $\alpha=\pi/3$ so that $T(\alpha)$ is equilateral. Recall $\mu_s(\alpha)=16\pi^2/9$ from Section~\ref{equilateral}. Note $w \circ \tau$ has mean value zero over $T(\alpha)$, and is symmetric, and so can be used as a trial function for $\mu_s(\alpha)$. Let $G(\beta)= \big( 4\sin^2(\beta/2) -1 \big)/3$. Since $\alpha=\pi/3 < \beta$, Lemma~\ref{lemcomp}(i) implies that if
\[
\kappa < \frac{1}{4} + \frac{(1/4)(3/4)}{\sin^2(\beta/2)-(1/4)} \frac{4\sin^2(\beta/2) - 1}{3} = \frac{1}{2}
\]
then
\[
\frac{16\pi^2}{9} = \mu_s(\pi/3) < \frac{4\sin^2(\beta/2) + 2}{3} \mu_s(\beta) ,
\]
which can be rewritten as
\begin{equation} \label{eq:rewrite}
\frac{16\pi^2}{12\sin^2(\beta/2)+6} < \mu_s(\beta) .
\end{equation}
Thus if $\kappa < 1/2$ then $\mu_1(\beta)<\mu_s(\beta)$ by \eqref{eq:mu1upper} and \eqref{eq:rewrite}, and so the fundamental mode of $T(\beta)$ must be antisymmetric.

Next take $\alpha=\pi/2$, so that $T(\alpha)$ is a right isosceles
triangle. Its smallest positive eigenvalue with symmetric
eigenfunction is $\mu_s(\pi/2)=2\pi^2$ (with eigenfunction
$\cos(\sqrt{2}\pi x)+\cos(\sqrt{2}\pi y)$, which gives the nodal domains in Figure~\ref{fig:symmmode}(b)). 

Note $w \circ \tau$ has mean value zero over
$T(\alpha)$, and is symmetric, and so can be used as a trial
function for $\mu_s(\alpha)$. Let $G(\beta)= \big(
6\sin^2(\beta/2) -1 \big)/4$. Notice $G(\beta)>0$ because
$\beta>\pi/3$. Putting $\alpha=\pi/2$ and $\beta<\pi/2$ into
Lemma~\ref{lemcomp}(ii), we see that if $\kappa \geq 1/2$ then
\[
2\pi^2 = \mu_s(\pi/2) < \frac{6\sin^2(\beta/2) + 3}{4} \mu_s(\beta) .
\]
This last expression is equivalent to \eqref{eq:rewrite}, completing the proof when $\kappa \geq 1/2$.

\section{\bf Proof of Theorem \ref{th:tld2}: the lower bound on $\mu_1 D^2$} \label{sec:bisec}

First we reduce to subequilateral triangles.
\begin{proposition} \label{pr:sharp}
Given any triangle, there exists a subequilateral or equilateral triangle of the
same diameter whose fundamental tone is less than or equal
to that of the original triangle. The inequality is strict,
unless the original triangle is itself subequilateral or equilateral.
\end{proposition}
\begin{proof}
Linearly stretch the given triangle in the direction perpendicular
to its longest side, until one of the other sides has the same
length as the longest side. This new triangle is isosceles by
construction, with the same diameter (\emph{i.e.}, longest side
length) as the original triangle. The new triangle is
subequilateral or equilateral, since its third side is at most as
long as the two equal sides.

We will show that the stretching procedure reduces
the fundamental tone, by assuming the longest side of the triangle lies along
the $x$-axis and applying the following general argument.

Let $\Omega$ be a planar Lipschitz domain. For each $t > 1$, let
$\Omega_t = \{ (x,ty) : (x,y) \in \Omega \}$ be the domain obtained
by stretching $\Omega$ by the factor $t$ in the $y$ direction. Given
any trial function $u \in H^1(\Omega)$ we have the trial function
$v(x,y)=u(x,y/t)$ in $H^1(\Omega_t)$, with Rayleigh quotient
\begin{align}
R[v] & = \frac{\int_{\Omega_t} \big( u_x(x,y/t)^2+
u_y(x,y/t)^2/t^2 \big) \, dA}{\int_{\Omega_t} u(x,y/t)^2 \, dA} \notag \\
& = \frac{\int_\Omega \big( u_x(x,y)^2+u_y(x,y)^2/t^2 \big) \,
dA}{\int_\Omega u(x,y)^2 \, dA} \notag \\
& \leq R[u] \label{eq:stretch}
\end{align}
since $t > 1$. In addition, if $u$ has mean value zero over $\Omega$, then so does $v$ over $\Omega_t$. Hence taking $u$ to be a fundamental mode $u_1$ for $\Omega$ implies that $\mu_1(\Omega_t) \leq \mu_1(\Omega)$, by the Rayleigh Principle.

The inequality of Rayleigh quotients in \eqref{eq:stretch} is strict
unless $u_y \equiv 0$. Thus the only possibility for the fundamental tone to
remain unchanged by the stretching is for the fundamental mode $u_1$ to
depend only on $x$. That cannot occur for a triangle, since
on sides of the triangle that are not parallel to the $x$-axis the Neumann boundary condition would force $\partial u_1 /\partial x \equiv 0$, so that $u_1 \equiv \text{(const.)}$ on the whole triangle, contradicting that $u_1$ is orthogonal to the constant mode.
Hence for triangles, the stretching procedure strictly reduces
$\mu_1$, when $t>1$.
\end{proof}

The point of Proposition~\ref{pr:sharp} is that when proving
Theorem~\ref{th:tld2}, we need only consider subequilateral and equilateral
triangles.

Recall the isosceles triangle $T(\alpha)$ with aperture $\alpha$
and side length $l$, and fundamental tone $\mu_1(\alpha)$. Assume
$0 < \alpha \leq \pi/3$, so that the triangle is subequilateral or
equilateral, with diameter $D=l$. Our task is to prove
\begin{equation} \label{eq:isos}
\mu_1(\alpha) > \frac{j_{1,1}^2}{D^2} , \qquad \alpha \in (0,\pi/3] .
\end{equation}
Equality holds asymptotically for degenerate acute isosceles triangles, since $\lim_{\alpha \to 0} \mu_1(\alpha) = j_{1,1}^2/D^2$ by Lemma~\ref{boundsiso}.

Numerical work suggests that $\mu_1(\alpha)$ is strictly increasing on
$(0,\pi/3]$, as shown in Figure~\ref{mustarfig}, but we have not been able to prove such monotonicity. Instead we bisect and stretch, as follows.

Cutting $T(\alpha)$ along its line of symmetry yields two right
triangles. Let $U(\alpha)$ be one of them.

The fundamental mode $v$ of $T(\alpha)$ is symmetric in the subequilateral case $\alpha \in (0,\pi/3)$, by Theorem~\ref{sharpsymmetric}, and it can be chosen to be symmetric in the equilateral case $\alpha=\pi/3$, by Section~\ref{equilateral}. Since $v$ has mean value zero over
$T(\alpha)$, it also has mean value zero over $U(\alpha)$. It follows from the Rayleigh Principle and symmetry that
\[
\mu_1 \big( U(\alpha) \big) \leq R_{U(\alpha)}[v] = R_{T(\alpha)}[v] = \mu_1(\alpha) .
\]

Now linearly stretch the right triangle $U(\alpha)$ in the direction
perpendicular to its longest side. After some amount of stretching, we
obtain a subequilateral triangle $T(\alpha_1)$ with the same diameter and with aperture determined by $\cos \alpha_1 = \cos^2 (\alpha/2)$, as some simple trigonometry reveals (see Figure~\ref{trigfig}).
Hence $\sin(\alpha_1/2) = \sin(\alpha/2)/\sqrt{2}$.

The stretching strictly reduces the fundamental tone, by Proposition~\ref{pr:sharp} and its proof, and so
\[
\mu_1(\alpha_1) < \mu_1 \big( U(\alpha) \big) \leq \mu_1(\alpha) .
\]

\begin{figure}[t]
  \begin{center}
\beginpgfgraphicnamed{triangles5_pic4}
\begin{tikzpicture}[scale=1]
      \pgfmathsetmacro{\x}{10}
      \pgfmathsetmacro{\y}{40}
      \coordinate (a) at (\x,0);
      \draw[dotted] (a) arc (0:\y:\x) coordinate (b);
      \fill[lightgray] (0,0) -- ($(a)!0.5!(b)$) coordinate (c)  {node[black,pos=0.8,below right=3pt] {$U(\alpha)$}} {node [below,pos=0.19,black] {$\alpha$}} {node [below,pos=0.36,black] {$\alpha_1$}} -- (a) -- cycle;
      \draw (0,0) -- (a) node[below,pos=0.5] {$D$}-- (b) -- (0,0) {node [above,pos=0.5] {$D$}} node [below right=3pt,pos=0.2] {$T(\alpha)$};
      \draw[loosely dashed] ($(0,0)!(c)!(a)$) coordinate (d) -- (c);
      \draw[scale=0.15] (\x,0) arc (0:\y:\x);
      \clip (0,0) -- (a) arc (0:\y:\x) -- cycle;
      \node (cir) at (0,0) [circle through=(a)] {};
      \draw[clip] (0,0) -- (intersection 2 of cir and c--d) coordinate (e) node [below right=2pt,pos=0.8] {$T(\alpha_1)$} -- (a);
      \draw[scale=0.3] (\x,0) arc (0:\y:\x);
      \draw[loosely dashed] (c) -- (e);
    \end{tikzpicture}
\endpgfgraphicnamed
  \end{center}
  \caption{Illustration of the bisection and stretching algorithm in Section~\ref{sec:bisec}, for the subequilateral triangle $T(\alpha)$ (which has been rotated here for clarity). } \label{trigfig}
\end{figure}
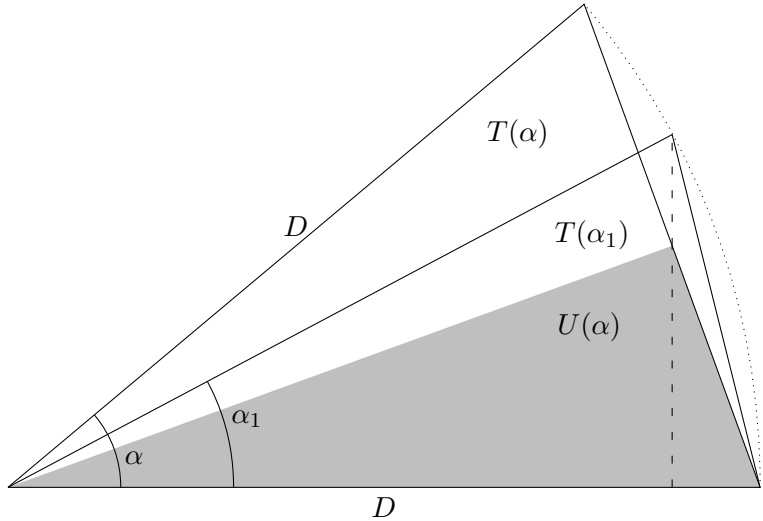

Continuing in this fashion, we deduce
\[
\mu_1(\alpha_n) < \cdots < \mu_1(\alpha_1) < \mu_1(\alpha)
\]
where the apertures $\alpha_n < \cdots < \alpha_1 < \alpha$ satisfy
\[
\sin(\alpha_n/2) = \frac{\sin(\alpha_{n-1}/2)}{\sqrt{2}} .
\]
Thus $\lim_{n \to \infty} \alpha_n = 0$, so that $\lim_{n \to
\infty} \mu_1(\alpha_n) = j_{1,1}^2/D^2$  by
Lemma~\ref{boundsiso}. Hence $(j_{1,1}^2/D^2) < \mu_1(\alpha)$,
which proves \eqref{eq:isos}.

\section{\bf Proof of Proposition \ref{1D}}
By rescaling and rotating we can suppose the superequilateral triangle is $T(\alpha)$ for some $\pi/3<\alpha<\pi$.

The upper bound in the Proposition is just Cheng's inequality \eqref{Cheng}, which was proved directly for superequilateral triangles in Lemma~\ref{chengsupereq}.

For the lower bound, first recall from Theorem~\ref{bluntantisymmetric} that
the fundamental mode $v$ of the superequilateral triangle $T(\alpha)$ is antisymmetric, and hence vanishes along the $x$-axis. Write $U(\alpha)$ for the upper half of $T(\alpha)$, so that $v$ satisfies a Dirichlet condition on the bottom edge of $U(\alpha)$. Let $z_+= \big( l\cos(\alpha/2),l\sin(\alpha/2) \big)$ be the upper vertex of $U(\alpha)$.

Consider the sector with center at $z_+$ and sides of length $l$ running from $z_+$ to the origin and from $z_+$ to $z_+ - (0,l)$, and with its arc running from the origin to $z_+ - (0,l)$. We are interested in the fundamental tone of the Laplacian on this sector, when Dirichlet boundary conditions are imposed on the arc and no conditions are imposed on the two sides. Defining $w$ to equal $v$ on $U(\alpha)$ and zero outside it, we see that $w$ is a Sobolev function in the sector and equals zero on the arc. Hence $w$ is a valid trial function for the fundamental tone of the sector. That fundamental tone equals $(j_{0,1}/l)^2$ (with fundamental mode $J_0(j_{0,1}|z-z_+|/l)$), and so
\[
\Big( \frac{j_{0,1}}{l} \Big)^{\! 2} \leq R[w] = R_{U(\alpha)}[v] = R_{T(\alpha)}[v] = \mu_1(\alpha) .
\]
Since $T(\alpha)$ has diameter $D=2l\sin(\alpha/2)$, the Proposition follows immediately.

\section*{Acknowledgments} We are grateful to Mark Ashbaugh and G\'{e}rard Philippin for guiding us to relevant parts of the literature.


\begin{thebibliography}{99}
\bibitem{AF06} P. Antunes and P. Freitas. {\it New bounds for the principal
Dirichlet eigenvalue of planar regions.} Experiment. Math. 15
(2006), no.~3, 333--342.

\bibitem{AF08} P. Antunes and P. Freitas. {\it A numerical study of the spectral gap.}
J. Phys. A  41 (2008), no. 5, 055201, 19 pp.

\bibitem{A99}
M. S. Ashbaugh. {\it Isoperimetric and universal inequalities for
eigenvalues.} Spectral theory and geometry (Edinburgh, 1998),
95--139, London Math. Soc. Lecture Note Ser., 273, Cambridge Univ.
Press, Cambridge, 1999.

\bibitem{AtBu} R. Atar and K. Burdzy. {\it On nodal lines of Neumann
eigenfunctions.} Electron. Comm. Probab. 7 (2002), 129--139.

\bibitem{B79} C. Bandle. Isoperimetric Inequalities and
Applications. Pitman, Boston, Mass., 1979.

\bibitem{BaBu} R. Ba\~nuelos\ and\ K. Burdzy. {\it On the ``hot spots'' conjecture of J.\ Rauch.} J. Funct. Anal. 164 (1999), no.~1, 1--33.

\bibitem{BF09}
D. Borisov and P. Freitas. {\it Singular asymptotic expansions for
Dirichlet eigenvalues and eigenfunctions of the Laplacian on thin
planar domains.} Ann. Inst. H. Poincar\'{e} Anal. Non Lin\'{e}aire,
to appear.

\bibitem{CF77}
I. Chavel and E. A. Feldman. {\it An optimal Poincar\'{e} inequality for convex domains of non-negative curvature.} Arch. Rational Mech. Anal. 65 (1977), no. 3, 263--273.

\bibitem{cheng} S. Y. Cheng. {\it Eigenvalue comparison theorems
and its geometric applications.} Math. Z. 143 (1975), no. 3,
289--297.

\bibitem{F06} P. Freitas. {\it Upper and lower bounds for the first Dirichlet
eigenvalue of a  triangle.} Proc. Amer. Math. Soc. 134 (2006),
no.~7, 2083--2089.

\bibitem{F07}
P. Freitas. {\it Precise bounds and asymptotics for the first
Dirichlet eigenvalue of triangles and rhombi.} J. Funct. Anal. 251
(2007), no. 1, 376--398.

\bibitem{fresiu} P. Freitas and B. Siudeja. {\it Bounds for the first Dirichlet
eigenvalue of triangles and quadrilaterals.} Preprint.

\bibitem{He06} A. Henrot. Extremum Problems for Eigenvalues of Elliptic Operators.
Frontiers in Mathematics. Birkh\"{a}user Verlag, Basel, 2006.

\bibitem{K06} S. Kesavan. Symmetrization \& Applications.
Series in Analysis, 3. World Scientific Publishing Co. Pte. Ltd., Hackensack, NJ, 2006.

\bibitem{LS09a} R. S. Laugesen and B. Siudeja.
{\it Maximizing Neumann fundamental tones of triangles.} Preprint. \url{www.math.uiuc.edu/~laugesen/}

\bibitem{L90}
L. Lorch.
{\it Monotonicity in terms of order of the zeros of the derivatives of Bessel functions.}
Proc. Amer. Math. Soc. 108 (1990), no.\ 2, 387--389. 

\bibitem{LR08}
Z. Lu and J. Rowlett. {\it The fundamental gap conjecture on
polygonal domains.} Preprint.

\bibitem{M02}
B. J. McCartin. {\it Eigenstructure of the equilateral triangle. II.
The Neumann problem.} Math. Probl. Eng. 8 (2002), no. 6, 517--539.

\bibitem{PW60}
L. E. Payne and H. F. Weinberger. {\it An optimal Poincar\'{e}
inequality for convex domains.} Arch. Rational Mech. Anal. 5 (1960),
286--292.

\bibitem{Pi80} M. A. Pinsky. {\it The eigenvalues of an equilateral triangle.} SIAM J. Math.
Anal. 11 (1980), no.~5, 819--827.

\bibitem{PS51}
G. P\'{o}lya and G. Szeg\H{o}. Isoperimetric Inequalities in
Mathematical Physics. Princeton University Press, Princeton, New
Jersey, 1951.

\bibitem{Pr98} M. Pr\'{a}ger. {\it Eigenvalues and eigenfunctions of the Laplace operator on an equilateral triangle.}  Appl. Math. 43 (1998), no.~4, 311--320.

\bibitem{S07} B. Siudeja. {\it Sharp bounds for
eigenvalues of triangles.} Michigan Math. J. 55 (2007), no.~2,
243--254.

\bibitem{S09}
B. Siudeja. {\it Isoperimetric inequalities for eigenvalues of triangles.} Preprint.

\bibitem{S96}
R. Smits. {\it Spectral gaps and rates to equilibrium for diffusions
in convex domains.} Michigan Math. J. 43 (1996), no. 1, 141--157.

\bibitem{S81}
R. P. Sperb.
Maximum Principles and Their Applications.
Mathematics in Science and Engineering, 157. Academic Press, Inc., New York, 1981.

\bibitem{W52}
G. N. Watson. A Treatise on the Theory of Bessel Functions. Second edition. University Press, Cambridge, 1952.

\end{thebibliography}
\end{document}